\documentclass[12pt]{amsart}
\usepackage{amsmath,amsfonts,amsthm,amssymb,enumerate,stmaryrd}
\usepackage{young}
\usepackage[margin=2.5cm]{geometry}
\usepackage[all]{xy}
\usepackage{listing}
\usepackage[normalem]{ulem}
\usepackage{color}
\usepackage{framed}
\usepackage{comment}
\usepackage{graphicx}
\definecolor{shadecolor}{gray}{0.875}
\renewcommand{\P}{\mathbb{P}}

\renewcommand{\O}{\mathcal{O}}
\newcommand{\R}{\mathbb{R}}
\newcommand{\I}{\mathcal{I}}
\newcommand{\C}{\mathbb{C}}

\newcommand{\ra}{\rightarrow}
\newcommand{\ol}{\overline}
\newcommand{\vol}{\mathrm{vol}}
\newcommand{\W}{W}
\newcommand{\J}{\mathcal{J}}

\def\PP{X}
\def\pp{X}

\def\Z{{\mathbb Z}}

\def\Kp1{{{\rm K}}_{p,1}(\mathbb P^n,b;d)}

\def\Bp1{{{\mathbb K}}_{p,1}(b;d)}

\def\vol{{\rm vol}}

\numberwithin{equation}{section} 
\newtheorem*{thm1}{Theorem}

\newtheorem{thm}{Theorem}[section]
\newtheorem{prop}[thm]{Proposition}
\newtheorem{lem}[thm]{Lemma}
\newtheorem{exap}[thm]{Example}

\theoremstyle{definition}

\newtheorem{rmk}[thm]{Remark}

\newtheorem{defn}[thm]{Definition}

\excludecomment{OmittedProofs}

\title{Asymptotic Weights of Syzygies of Toric Varieties}
\author{Xin Zhou}
\begin{document}

\maketitle

\section{Introduction}

The purpose of the paper is to give a sharp asymptotic description of the weights that appear in the syzygies of a toric variety. We prove that as the positivity of the embedding increases, in any strand of syzygies, torus weights after normalization stabilize to the same fixed shape that we explicitly specify. 

Let $X$ be a projective toric variety over $\C$ of dimension $n$, and $L$ be a very ample toric line bundle on $X$. Then $L$ defines a toric embedding:
\[X \hookrightarrow \P^{r(L)}  = \P H^0(X,L) = \textnormal{Proj} \  S\]
where $r(L) = h^0(X,L) - 1$ and $S = \textnormal{Sym} H^0(X,L)$.
Write:
\[R(X;L) = \bigoplus_m H^0(X, mL )\]
which is viewed as a finitely generated graded $S$-module. We will be interested in the syzygies of $R(X;L)$ over $S$. Specifically, $R$ has a graded minimal free resolution
\[\mathbb{F}: ... \ra F_p \ra ... \ra F_0 \ra R \ra 0\]
where $F_p = \oplus_j S(-a_{p,j})$ is a free $S$-module.  Write $K_{p,q}(X;L)$ for the finite dimensional vector space of minimal $p$-th syzygies of degree $(p+q)$, so that:
\[F_p \cong \bigoplus_q K_{p,q}(X;L) \otimes_{\mathbb{C}} S(-p-q)\] 
Moreover, in the above setting, the torus action on $X$ induces torus actions on $K_{p,q}(X;L)$. We can naturally ask which torus weights appear in their decompositions. 

From an asymptotic perspective, Ein and Lazarsfeld show in \cite{EL} that for $1 \leq q \leq n$, if $L \gg 0$, $K_{p,q}(X;L) \neq 0$ for almost all $p \in [1, r_d]$. 
In this paper, we give a sharp description of the asymptotic distribution of normalized torus weights in syzygies. To give the statement, let $\Delta$ be the convex polytope associated to the very ample divisor $A$ (\cite{Ft}, Section 3.4, p66, $P_A$ in notation of the book.) Let $L_d = A^{\otimes d}$. Then by degree counting, the torus weights of $K_{p,q}(\PP;L_d)$ correspond to integral points in $(p+q)d \cdot \Delta$. Denote the collection of weights by:
\[  \mathrm{wts}(K_{p,q}(\PP;L_d)) = \{ \mathrm{Torus \ weights \ of \ }K_{p,q}(\PP;L_d) \} \subseteq (p+q)d \cdotp \Delta\] 
We normalize so that all points lie in $\Delta$: 
\[  \mathrm{wts^{nor}} (K_{p,q}(\PP;L_d)) = \frac{\mathrm{wts}(K_{p,q}(\PP;L_d)) }{(p+q)d} \subseteq \Delta\]

\noindent We show the philosophy that asymptotic syzygies are complicated by proving that as $d \ra \infty$, the set of all normalized torus weights becomes dense in $\Delta$:
\begin{thm1} [\ref{dense}]
Fix $1 \leq q \leq n$, then
\begin{equation*}
\bigcup_{\substack{d \ > \ 0 \\ 1 \ \leq \ p \ \leq \  r_d}} \mathrm{wts^{nor}}(K_{p,q}(\PP;L_d))
\end{equation*}
is dense in $\Delta$.
\end{thm1}
The Theorem is illustrated in Figure \ref{dense weights}, which approximates $K_{p,1}(\PP^2,d)$ for $d = 2$ and $d = 4$. 
\bigskip
\begin{figure}[here]
\includegraphics[scale = 0.32]{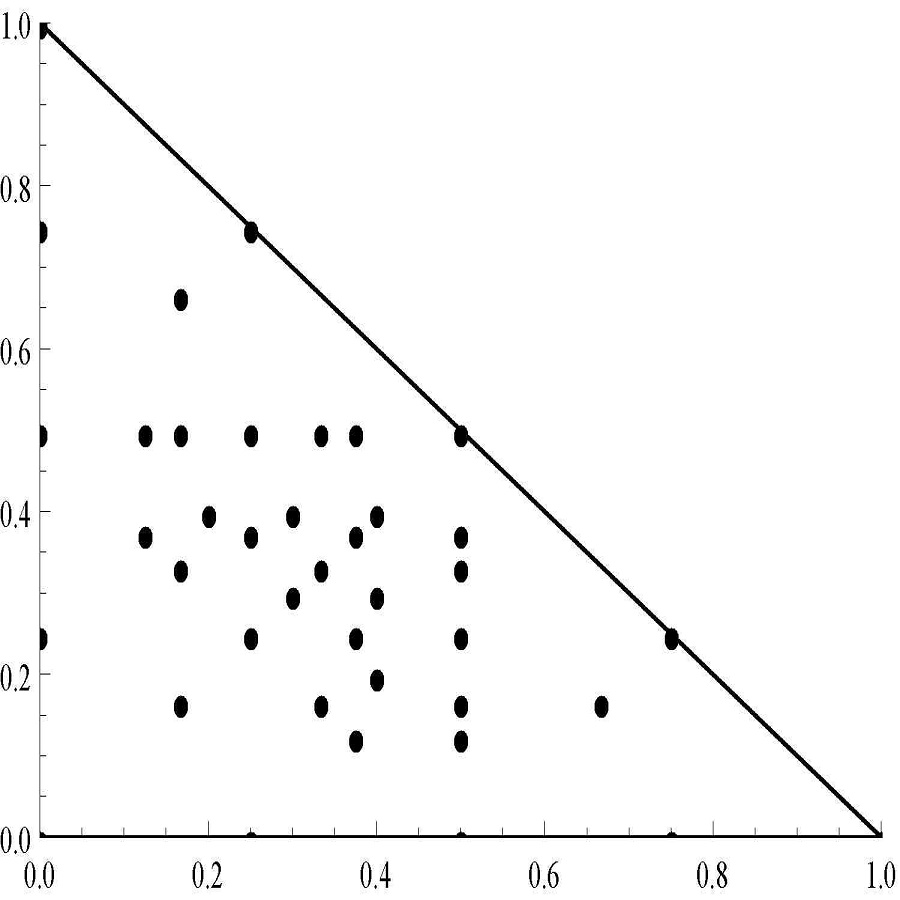} \quad \quad \quad
\includegraphics[scale = 0.32]{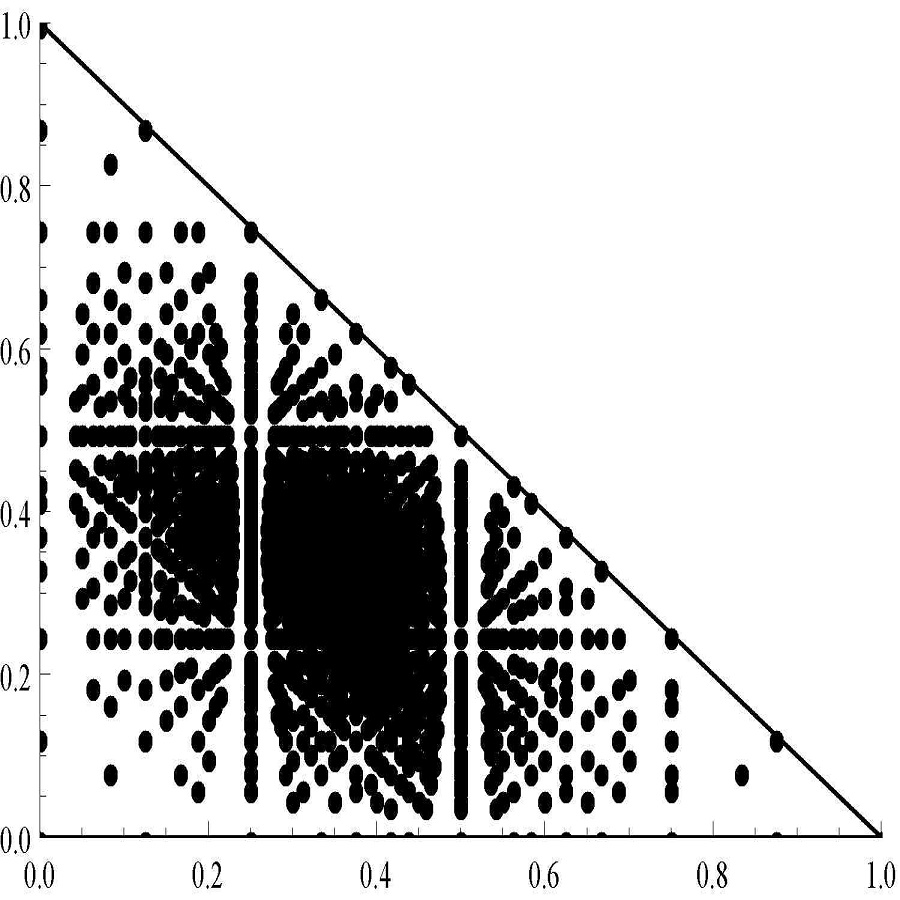} 
\caption{Normalized torus weights for $K_{p,1}( \PP^2;d)$ for d = 2 and d = 4.}
\label{dense weights}
\end{figure}

We can also ask what happens if we focus only on some of the syzygies appearing in the resolution. Is their behavior still as complicated? More specifically, restrict $p$ to lie in a fixed interval relative to $r_d$, i.e. consider:
\[ \Delta(a, b) \  = \overline{\bigcup_{\substack{d \  \gg \  0 \\ a \cdot r_d \ \leq \  p \ \leq \  b \cdot  r_d}} \mathrm{wts^{nor}}(K_{p,q}(\PP;L_d))}\subseteq \ \Delta \]
where $0 \leq a < b \leq 1$. These are no longer necessarily dense inside all of $\Delta$. Figure \ref{not dense} shows the normalized weights of $K_{p,1}(\PP^2;4)$ for $a = 0.33,b = 0.66$ and $a = 0.66, b = 1$:

\begin{figure}[here]
\includegraphics[scale = 0.32]{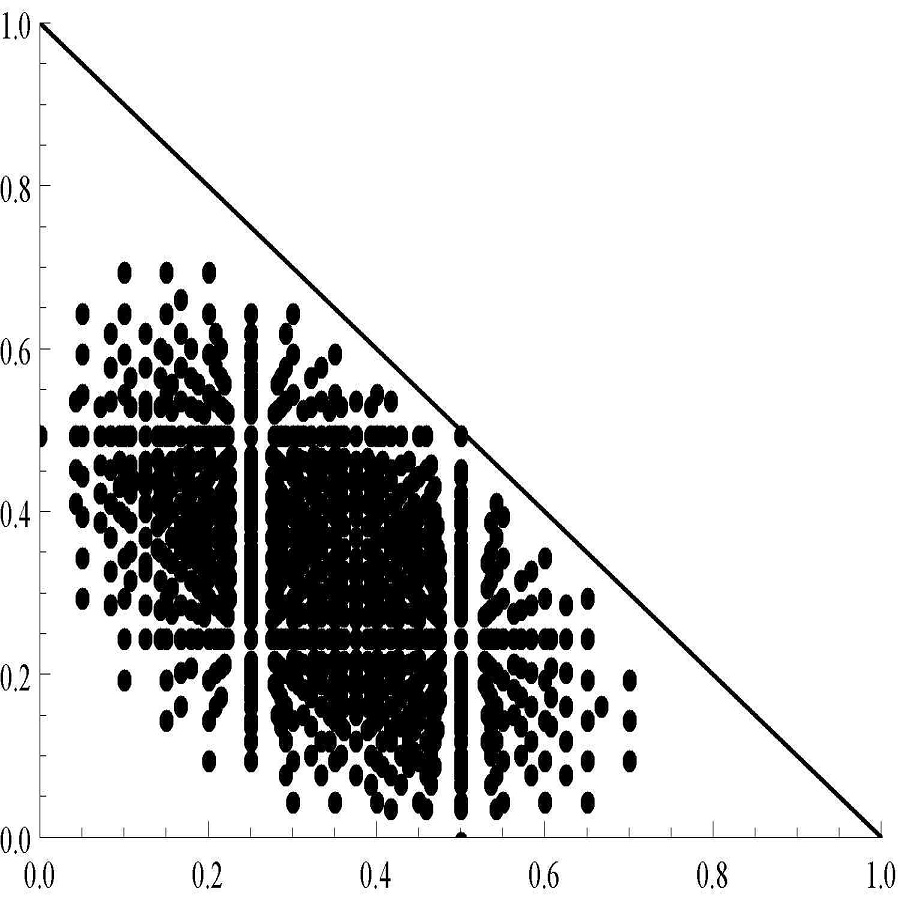} \quad \quad \quad
\includegraphics[scale = 0.32]{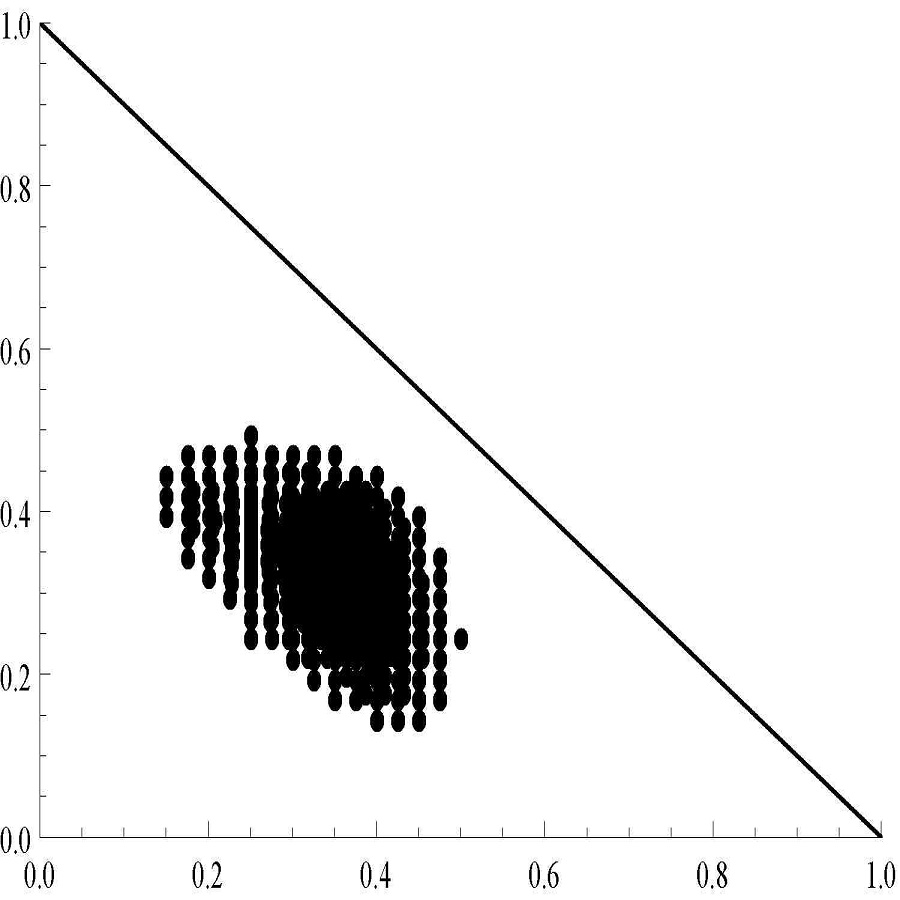} 
\caption{Closure of normalized weights for $a = 0.33,\  b = 0.66$ and $a = 0.66, \ b = 1$ with $\PP = \PP^2, \ d = 4$.}
\label{not dense}
\end{figure}

Quite surprisingly, we can explicitly describe this set. The description involves the largest volume of a "nice" region supported at $x$. More precisely, define:
$$\tau_x = \mathrm{sup} \ \{ \vol(S_x)\  | \ S_x = \mathrm{\ finite \  union \  of \  cubes} \subset \Delta, \mathrm{\ and \ center \ of \ mass \ of} \ S_x = x\}$$
\begin{thm1}[\ref{delta a b}] One has:
$$\Delta(a,b)  = \left \{x \in \Delta \  \bigg| \  \frac{\tau_x}{\vol(\Delta)} \geq a \right \} =: \Delta(a)$$ 
\end{thm1}

Note that part of the statement of the theorem is that $\Delta(a,b)$ does not depend on $b$, so we write $\Delta(a)$ for it. The boundary of $\Delta(a)$ is also explicitly computable. For example, let $\Delta$ be the unit square. Then boundary of $\Delta(\frac{1}{10})$ consists of 12 pieces, 4 segments of hyperbolas at the corners and 8 line segments in between as illustrated below. 

\begin{figure}[here]
\includegraphics[scale = 0.50]{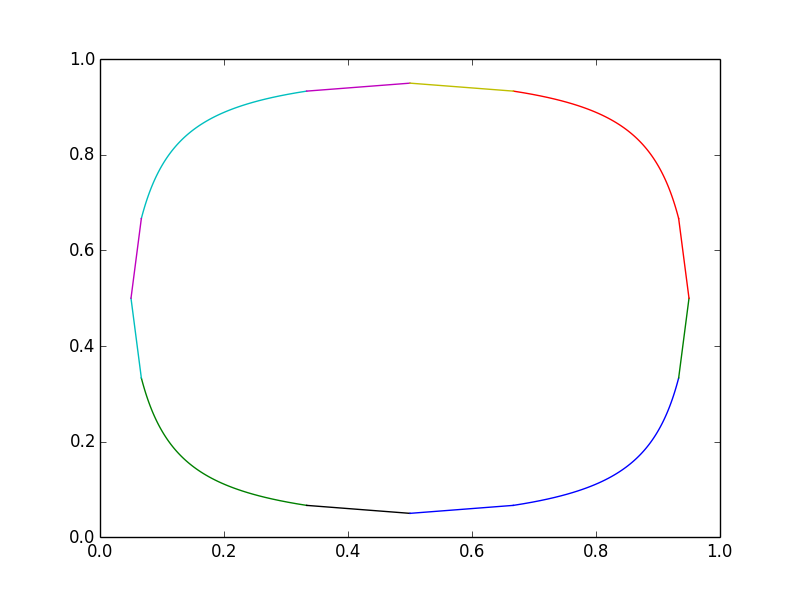}
\caption{$\Delta(\frac{1}{10})$ for the unit square.}
\label{dense weights}
\end{figure}

\newpage

In order to orient the reader, for the rest of the introduction we discuss at some length, the basic strategy of the proof. In the body of the paper, which gives full details, we will refer back to this preview as a roadmap. 

Let $L$ be a very ample toric line bundle on a smooth projective toric variety $X$. As in \cite{GL}, \cite{L89} and \cite{EL}, for $L$ in the evaluation map:
$$\nu_L: H^0(X, L) \otimes_\C \O_X \ra \O_X(L)$$
we put $M_L = \mathrm{ker} \ \nu_L$. 
Thus $M_L$ is a vector bundle sitting in the basic exact sequence:
\begin{equation}\label{ML}
0 \ra M_L \ra H^0(X, L) \otimes_\C \O_X \ra \O_X(L) \ra 0
\end{equation}
Thanks to Demazure vanishing, we have Prop. \ref{cohomology}:
\begin{equation}\label{idea of proof cohomology}
K_{p,q}(\pp;L) = H^q(X, \wedge^{p+q}M_{L}).
\end{equation}
So the issue is to identify the torus weights appearing in the right hand side of the equality.

To a first approximation, the idea is to find torus equivariant spaces 
$$U, \quad W_1, \quad \mathrm{with} \ \dim U = 1, \quad \dim W_1 \gg 0$$
together with a torus equivariant map:
\begin{equation}\label{idea of proof surjective map}
H^q(X, \wedge^{p+q}M_{L}) \longrightarrow U \otimes \wedge^{p+q} W_1.
\end{equation}
Suppose one knew that (\ref{idea of proof surjective map}) is surjective. Then we can conclude that every weight appearing in $U \otimes \wedge^{p+q}W_1$ appears in:
$$K_{p,q}(\pp;L) = H^q(X, \wedge^{p+q}M_{L}).$$
On the other hand, one can compute combinatorially the weights of $U \otimes \wedge^{p+q}W_1$ from the weights of $U$ and $W_1$, and the results stated in the previous section would follow. 

Strictly speaking, we do not achieve proving surjectivity of (\ref{idea of proof surjective map}). What we show is that we can find torus stable vector spaces $W_0$ of small dimension and $W_1$ of large dimension with $W_0$ a quotient of $W_1$ with the following property. Let $W$ be any quotient of $W_1$ that factors the map to $W_0$. 
$$W_1 \xrightarrow\!\!\!\!\to W \xrightarrow\!\!\!\!\to W_0.$$
Then there is a surjective mapping:
$$H^q(X, \wedge^{p+q}M_{L}) \longrightarrow U \otimes \wedge^{p+q} W$$
with $\dim W = p+q$. As before, this allows us to produce many weights appearing in $K_{p,q}$ and the stated theorem follows. 

The next point is to understand how to construct $U$, $W_0$ and $W_1$. Take a $w$-dimensional torus stable quotient of $H^0(X,L)$ and denote it by $W$. $W$ defines a toric stable linear subspace:
$$\P(W) \subset \P(H^0(L)).$$
Let $Z \subset X$ be the scheme theoretic intersection:
\begin{equation}\label{idea of proof W intersect X}
Z = \P(W) \cap X.
\end{equation}
Then there is a natural map:
$$W \otimes_\C \O_X \longrightarrow L \otimes \O_Z.$$
Taking wedge powers, a local analysis (cf. (\ref{sigma pi diagram})) shows that there is a surjective homomorphism:
\begin{equation*}
\wedge^w M_L \longrightarrow \I_{Z/X} \otimes \wedge^wW.
\end{equation*}
Hence, we have a map:
\begin{equation}\label{idea of proof cohomology of ideal sheaf of Z}
H^q(X,\wedge^w M_L) \longrightarrow H^q(X,\I_{Z/X}) \otimes \wedge^w W
\end{equation}
which is also toric equivariant. The goal is to choose $W$ such that:
\begin{equation}\label{U not equal to 0}
U = H^q(X,\I_{Z/X}) = \C \neq 0.
\end{equation}

In practice, (\ref{idea of proof cohomology of ideal sheaf of Z}) is achieved by first choosing a toric stable subspace $Z \subset X$ such that (\ref{U not equal to 0}) holds, and then choosing $W$ to satisfy (\ref{idea of proof W intersect X}). This is carried out in Section 3. Furthermore, we will see in Section 4 that we can take as $W$ quotients of a fixed very large $W_1$ (Prop. \ref{weights}).

The main technical result of Section 3 is that when we follow this outline, the resulting map (\ref{idea of proof cohomology of ideal sheaf of Z})
is surjective (Proposition \ref{carry syzygies}). One key point here is that although the map (\ref{idea of proof cohomology of ideal sheaf of Z}) is toric, to prove that it is surjective, we do not need to stay in the toric world. Hence, we can follow the inductive arguments in \cite{EL} and \cite{Z} with essentially no modification. 

There is one further asymptotic ingredient. Namely, we are interested in the asymptotics of $K_{p,q}(X;L_d)$ where $L_d = A^{\otimes d}$ and $A$ is very ample. When we go through the constructions just outlined for $L_d$, we arrive at the following situation. 

We have torus stable subspaces $W_{0,d}, W_{1,d}$, $W_d$:
$$\mathrm{weights}(\wedge^{p+q}W_d) + (\mathrm{some \ fixed \ weight}) \subset \mathrm{weights}(K_{p,q}(X;L_d)).$$ 
Moreover, 
$$\dim W_d = p+q, \ 
\mathrm{and} \ W_{1,d} \xrightarrow\!\!\!\!\to W_d \xrightarrow\!\!\!\!\to W_{0,d}.$$
Furthermore, 
$$\dim (W_{0,d}) \in o(d^n), \ \mathrm{and} \ h^0(X,L_d) - W_{1,d} \in o(d^n).$$
Thus, up to asymptotically insignificant contributions, all the weights of $\wedge^{p+q}W_{1,d}$ appear in $K_{p,q}$. It remains to prove a lemma on asymptotics of normalized weights for wedge powers. This is the content of Section 5. The asymptotic behavior we deduce applies to any sequence of toric quotient spaces $W_d$ asymptotically equal to $H^0(X,L_d)$ in dimension. Then specializing to $W_{1,d}$ gives us a lower bound on the weights that appear by the above discussion. Applying the result to $H^0(X,L_d)$ gives us an upper bound by the definition of Koszul cohomology. Hence, we get our sharp asymptotic description. 

I would like to thank my advisor, Rob Lazarsfeld, for his continued suggestions and support. I would also like to thank Daniel Erman, David Speyer and Linquan Ma for helpful discussions and fruitful conversations.

\section{Surjectivity of map induced by a secant space}

In this section, we adapt the computations in \cite{EL} and \cite{Z} to the toric case. The reader who is not familiar with the argument in these papers might find it helpful to read the outline of the proof appearing in the end of the previous section.

\subsection{Key lemma} We first recall the key vector bundle used to compute syzygies. Let $X$ be a smooth projective toric variety over $\C$. Let $A$ be a fixed toric very ample line bundle on $X$. We use $L$ to denote any toric very ample line bundle on $X$ (we will later replace $L$ with $L_d = A^{\otimes d}$). 

As in \cite{GL}, \cite{L89} and \cite{EL}, in the evaluation map:
$$\nu_L: H^0(X, L) \otimes_\C \O_X \ra L$$
we put $M_L = \mathrm{ker} \ \nu_L$. Thus $M_L$ is a vector bundle sitting in the basic exact sequence:
$$ 0 \ra M_L \ra H^0(X, L) \otimes_\C \O_X \ra L \ra 0.$$
We will need the following fact in this chapter:
\begin{prop}\label{Demazure}(Demazure)
For any projective toric variety $X$, and a very ample divisor $A$, one has: 
$$H^m(X,\O_X(jA)) = 0 \  \mathrm{for} \ m \geq 1, j \geq 0.$$
\end{prop}
\begin{proof}
This follows from Demazure vanishing (cf. Thm 9.2.3 \cite{CLS}).
\end{proof}
In our setting $L_d = A^{\otimes d}$, we have:
\begin{prop}\label{cohomology}
For $1 \leq q \leq n$, $K_{p,q}(\pp;L_d) = H^q(X, \wedge^{p+q}M_{L_d})$.
\end{prop}
\begin{proof}
The conclusion follows as in \cite[Prop. 1.1]{Z}  and \cite[Prop. 3.2, Prop. 3.3]{EL} if we know: 
$$H^i(\pp,\O_\pp(mL_d)) = 0 \ \mathrm{for} \  i> 0, m \geq 0.$$ 
This follows from the Proposition above.
\end{proof}

Let $W$ be a quotient of $H^0(X,L)$ of dimension $w$. Then we have 
$$\P(W) \subset \P(H^0(X,L)).$$
Let
$$Z = \P(W) \cap X$$ 
the scheme-theoretic intersection of $\P(W)$ with $X$. This gives rise to a surjective map of sheaves:
$$ W_X = W \otimes \O_X \longrightarrow L \otimes \O_Z,$$
and we denote its kernel by $\Sigma_W$. So we get an exact diagram of sheaves:
\begin{equation}\label{first diagram}
\begin{gathered}
\xymatrix{
0 \ar[r] & M_L \ar[r] \ar[d] & V \otimes \O_X \ar[r] \ar[d] & L \ar[r] \ar[d] & 0 \\
0 \ar[r] & \Sigma_W \ar[r] & W \otimes \O_X \ar[r] & L \otimes \O_Z \ar[r] & 0 \\
}
\end{gathered}
\end{equation}

Through the local analysis of \cite{EL} (3.10), we get a diagram :
\begin{equation}\label{sigma pi diagram}
\begin{gathered}
\xymatrix{
\bigwedge^w \Sigma_W \ar[r] \ar[d] & \bigwedge^w W_X \ar[d] \\
\I_{Z/X} \otimes \bigwedge^w W_X \ar[r] &  \O_X \otimes \bigwedge^w W_X\\
}
\end{gathered}
\end{equation}
and this induces a surjective map (cf. the map above \cite{EL} Def. 3.8): 
\begin{equation}\label{sigma}
\sigma_\pi: \wedge^w M_L \rightarrow \I_{Z/X} \otimes \wedge^wW_X
\end{equation}
Then $\sigma_\pi$ induces a map:
\begin{equation}\label{sigma induced}
H^q(X,\wedge^w M_L) \rightarrow H^q(X,\I_{Z/X}) \otimes \wedge^w W
\end{equation}

The above works in general without any toric hypothesis. In our setting, when $X$, $L$, $W$ are toric, all the above maps are toric equivariant. Following the notations above, the key conclusion of this section is the following lemma:
\begin{lem}\label{weight inclusion}
For $L = L_d$ with $d \gg 0$ and $1 \leq q \leq n$, there exists a torus stable quotient $W$ with $Z = \P(W) \cap X$ and 
$$H^q(X,I_{Z/X}) \neq 0,$$ such that the induced torus equivariant map:
$$H^q(X,\wedge^w M_L) \rightarrow H^q(X,\I_{Z/X}) \otimes \wedge^w W$$
where $w = \dim W$, is surjective.  Therefore,  any torus weight in $H^q(X,\I_{Z/X}) \otimes \wedge^w W$ also appears in $K_{w-q,q}(X;L)$.
\end{lem}

In Section 2, we will show that there are many choices of $W$, giving many toric weights in $K_{w-q,q}(X;L_d)$.

\subsection{Proof of Lemma \ref{weight inclusion}}

Once torus equivariance has been stablished as in (\ref{sigma induced}), surjectivity has nothing to do with the torus action. So we will be able to prove the surjectivity of (\ref{sigma induced}) by proving the surjectivity of:
\begin{equation}\label{sigma pi}
H^q(X,\wedge^w M_L) \rightarrow H^q(X,\I_{Z/X}).
\end{equation}
The rest of the proof is technical and follows the same lines of attack as in \cite{Z} and \cite{EL}. We will give the choice of $Z$ and $W$ later (Lemma \ref{complete intersection}, Prop. \ref{carry syzygies}). For now, we introduce some terminology that will help in the induction. 

For induction in the proof, we have to add in a twist of the map in \ref{sigma pi}. Let $B$ be a line bundle and consider:
\begin{equation}\label{sigma pi with B}
H^q(X,\wedge^w M_L(B)) \rightarrow H^q(X,\I_{Z/X}(B))
\end{equation}
\begin{defn}\label{definition of carrying syzygies}
\textnormal{
Let $W$ be a quotient of $H^0(X,L)$ as above. We say that $W$ \textit{carries weight $q$ syzygies} for $B$ if the map induced by $\sigma_\pi$ in equation (\ref{sigma pi with B}) is surjective. (We also say the same for $q = 0$ for notational convenience even though it isn't necessarily directly related to syzygies.)
}
\end{defn}

Let us set up some inductive notation. Take a general divisor $\ol{X} \in |A|$ so that $\ol{X}$ is irreducible and diagram (\ref{first diagram}) remains exact after tensoring with $\O_{\ol{X}}$. For $0 \leq i \leq q-1$, let 
$$X_0 = X, \quad Z_0 = Z, \quad A_0 = A.$$ 
Having made the definitions for $i-1$, choose a general $X_i \in |A_{i-1}|$ so that $X_i$ is irreducible and the corresponding diagram (\ref{first diagram}) for $X_{i-1}$ remains exact after tensoring with $\O_{X_i}$ (and as previously defined, $\ol{X} = X_1$). Let 
$$Z_i = Z_{i-1} \cap X_i, \quad A_i = A_{i-1}|_{X_{i-1}}.$$

Now we construct a toric $Z$ in our smooth projective toric variety $X$ that satisfy the good properties in Definition \ref{adapted}. Let $-K_X = e_1+ ... + e_m$ where $\{e_i\}$ are the prime toric invariant divisors. Let $c = n+1 - q$ with $ 1 \leq q \leq n$.

\begin{lem}\label{complete intersection}
We can order the $e_i$ such that $Z = e_1 \cap ... \cap e_{c-1} \cap (e_c + ... + e_m)$ is a complete intersection. 
\end{lem}

\begin{proof}
Choose $e_1,...,e_n$ such that they generate an n-dimensional cone. Then $F = e_1 \cap ... \cap e_{c-1}$ is a complete intersection. For any $i>c-1$, $F$ either does not meet $e_i$, or it does so transversely since adding a ray to a cone increase its dimension by at most 1. It meets at least one of them, $e_c$, since $c \leq n$.
\end{proof}

Next, we establish a number of properties of $Z$. 
\begin{prop}\label{not zero cohomology}
With the above choice of $Z$:
$$H^q(X,I_{Z/X}) = \mathbb{C} \neq 0$$
\end{prop}

\begin{proof}
If $q = 1$, then $c = n$, and in this case $Z$ consists of two points. The short exact sequence 
$$0 \ra \I_{Z/X} \ra \O_X \ra \O_Z \ra 0$$
induces:
$$ H^0(\O_X) \ra H^0(\O_Z) \ra H^1(\I_{Z/X}) \ra H^1(\O_X)$$
where $h^0(\O_X) = 1$, $h^0(\O_Z) = 2$ and $h^1(\O_X) = 0$ (since structure sheaves of toric varieties do not have higher cohomology). Hence, $H^1(\I_{Z/X}) = \C$. 

Assume $q \geq 2$. From
$$ 0 \ra \I_{Z/X} \ra \O_X \ra \O_Z \ra 0$$
we get:
$$0 = H^{q-1}(\O_X) \ra H^{q-1}(\O_Z) \ra H^q(\I_{Z/X}) \ra H^q(\O_X) = 0.$$
Then $H^{q-1}(\O_Z) = H^q(\I_{Z/X})$. 

Let $F = e_1 \cap e_2... \cap e_{c-1}$. From
$$ 0 \ra \I_{Z/F} \ra \O_F \ra O_Z \ra 0.$$
we get:
$$0 = H^{q-1}(\O_F) \ra H^{q-1}(\O_Z) \ra H^q(\I_{Z/F}) \ra H^q(\O_F) = 0.$$
we need to compute $H^q(\I_{Z/F})$. Now $Z = F \cap (e_{c} + ...+ e_m)$ and $e_{c} + ... + e_m = -K_F$. Since
$$\I_{Z/F} = \O_F(K_F), \quad \dim F = n - {c-1} = n - (n - q) = q,$$
and then Serre duality applies and we have:
$$H^q(\I_{Z/F}) = H^0(\O_F) = \C.$$
\end{proof}

\begin{prop}
For all $m \geq 1, j \geq i \geq 0 $:
\begin{enumerate}\label{x vanishing}
\item $H^m(X_i, \O_{X_i}(jA)) = 0$.
\item $H^m(Z_i, \O_{Z_i}(jA)) = 0$.
\end{enumerate}
\end{prop}

\begin{proof}
We prove the first assertion by induction on $i$. When $i = 0$, the conclusion follows from Demazure vanishing since $X$ is toric. Suppose the conclusion is true for $i-1$, then we have:
$$0 \ra \O_{X_{i-1}}((j-1)A) \ra \O_{X_{i-1}}(jA) \ra \O_{X_i}(jA) \ra 0$$
$H^m(\O_{X_{i-1}}(jA)) = H^{m+1}(\O_{X_{i-1}}((j-1)A)) = H^{m+1}( O_{X_{i-1}}(jA)) = 0$ by inductive assumption, hence $H^m(\O_{X_i}(jA)) = 0$.
The second assertion is analogous. 
\end{proof}

\begin{prop}\label{zero cohomology}
For all $i \geq 0$:
$$H^{q-i}(X_i,\I_{Z_i/X_i}((i+1)A)) = 0$$
\end{prop}

\begin{proof}
Consider on $X_i$ the exact sequence:
$$0 \ra \I_{Z_i/X_i}((i+1)A) \ra \O_{X_i}((i+1)A) \ra \O_{Z_{i}}((i+1)A) \ra 0$$
giving rise to:
$$H^{q-i- 1}(\O_{X_{i}}((i+1)A)) \ra H^{q-i-1}(\O_{Z_{i}}((i+1)A)) \ra$$
$$ H^{q-i}(\I_{Z_i/X_{i}}((i+1)A)) \ra H^{q-i}(\O_{X_{i}}((i+1)A))$$
If $q-i = 1$, then $Z_i$ has dimension $\dim Z - i = \dim Z - (q-1) = n - (n+1-q) - (q-1) = q-1 - q+1 = 0$. Then very ampleness and $H^{q-i}(\O_{X_{i}}((i+1)A)) = 0$ from Proposition \ref{x vanishing} implies that $H^{q-i}(X_{i},\I_{Z_i/X_{i}}((i+1)A)) = 0$. Assume $q -i - 1 \geq 1$, then the two ends in the above sequence are 0 because of Proposition \ref{x vanishing} and we get:
$$H^{q-i}(\I_{Z_i/X_{i-1}}((i+1)A)) = H^{q-i-1}(\O_{Z_{i}}((i+1)A)) = 0,$$
as claimed. 
\end{proof}
\begin{defn}\label{adapted}
We say that $Z$ is \textit{adapted} to the data $X, B, A, n, q$, if:
\begin{enumerate}
\item $H^q(X,\I_{Z/X}(B)) \neq 0$
\item For all $i \geq 0$, $H^{q-i}(X_i,\I_{Z_i/X_i}(B+ (i+1)A))) = 0$.
\item For all $i \geq 0$, $Z_i$ has dimension $q-1 -i$. 
\end{enumerate}
\end{defn}

Putting the computations of cohomologies of $Z$ together, we obtain:
\begin{prop}
For any $1 \leq q \leq n$, the scheme $Z$ constructed above is adapted to $X, \O_X, A, n, q$
\end{prop}
\begin{proof}
Choosing the divisors as in Proposition \ref{complete intersection} , Definition \ref{adapted} (iii) follows from the complete intersection condition. Definition \ref{adapted} (i), (ii) are checked in Proposition \ref{not zero cohomology}, \ref{zero cohomology}.
\end{proof}

Having constructed $Z$, we next turn to the construction of quotients $W$ as in Definition \ref{definition of carrying syzygies}. The issue is to specify inductive conditions that will guarantee that the condition in that definition holds. Recall that 
$$V = H^0(X,L).$$ 
Let 
\begin{equation}\label{v prime}
V' = V \cap H^0(X,I_{\ol{X}/X}(A)).
\end{equation}
The intersection takes place inside $V$. Set $W' = \pi(V')$. Write
\begin{equation}\label{bar}
\ol{V} = V/V', \ \ol{W} = W/W', \ \ol{L} = L |_{\ol{X}}, \ \ol{B} = B |_{\ol{X}}, \ \ol{Z} = Z \cap \ol{X}
\end{equation}

As in \cite[(3.14)]{EL}, we get the analogue of (1.3) above for the barred objects and we have the surjection:
\[\ol{\sigma}: \wedge^{\ol{w}} M_{\ol{V}} \ra I_{\ol{Z}/\ol{X}},\]
so we can study the behavior of $\ol{W}$ with respect to carrying syzygies. 

\begin{lem}\label{inductive lemma}
Fix $1 \leq q \leq n$. If $\ol{W}$ carries weight $q-1$ syzygies for $\ol{B} + \ol{A}$ on $\ol{X}$ and if
\[H^q(X,I_{Z/X}(B+A)) = 0,\]
then $W$ carries weight q syzygies for $B$ on $X$.
\end{lem}

\begin{proof}
This follows from the same argument as \cite{EL} Thm 3.10 with $(q-1)$ replaced by $q$ and $B \otimes L$ with $B$ in our case.  
\end{proof}

\begin{prop}\label{surjective}
If $d \gg 0$, then the following statements are true and so are their inductive counterparts after cutting down by hyperplanes as above:
\begin{enumerate} 
\item The map $H^0(X,L_d) \ra H^0(Z,L_d)$ is surjective; equivalently: $$H^1(X,I_{Z/X}(L_d)) = 0.$$
\item The map $H^0(Z,L_d)  \ra H^0(\ol{Z},L_d)$ is surjective; equivalently: $$H^1(Z,L_d - A) = 0.$$
\item $H^1(X,I_{Z/X}(L_d-A)) = 0$ (or equivalently, with $W'$ chosen below, the map $V' \ra W'$ is surjective.)
\item The map $H^0(X,L_d)  \ra H^0(\ol{X}, L_d)$ is surjective, or equivalently $$H^1(X,L_d-A) = 0.$$
\item $I_{Z/X} \otimes \O_X(L_d)$ is globally generated.
\end{enumerate}
\end{prop}
\begin{proof}
These all follow from Serre vanishing.
\end{proof}

\begin{prop}\label{carry syzygies}
Fix $1 \leq q \leq n$. Suppose there exists a subscheme $Z$ of $X$ adapted to $X, B, A, n, q$. Take $W_{d} = H^0(X,\O_Z(L_d))$. Then for $d \gg 0$, $W_{d}$ carries weight $q$ syzygies for $B$.
\end{prop}
\begin{proof}
Start with $W = W_d = H^0(Z,L_d)$. By the definitions in equation (\ref{v prime}), (\ref{bar}) and surjectivity from Prop. \ref{surjective}, we have: 
\begin{equation}
\xymatrix{
0 \ar[r] &V'= H^0(X,L_d-A) \ar[r] \ar[d] &  V = H^0(X,L_d) \ar[r] \ar[d] & \ol{V} = H^0(\ol{X},L_d) \ar[d] \ar[r] & 0\\ 
0 \ar[r] & W' \ar[r] \ar[r] \ar[d] & W = H^0(Z,L_d) \ar[r] \ar[d] &  \ol{W} \ar[d] \ar[r] & 0\\
& 0  & 0 & 0\\
}
\end{equation}
where 
\begin{equation}
W' = H^0(Z,L_d-A), \ \ol{W} = H^0(\ol{Z},L_d).
\end{equation}
The sheaf $I_{Z/X} \otimes \O_X(L_d)$ is globally generated (Remark \ref{surjective} (v)) so $Z = \P(W) \cap X$. Moreover, when we cut down by hyperplanes as in Lemma \ref{inductive lemma}, we obtain the corresponding diagrams in lower dimensions. 

We prove the Proposition by induction on $q$. $Z$ is always of dimension $q-1$. When $q = 1, Z$ consists of points. If $X$ is of dimension 1, then surjectivity follows from the fact that sheaf surjective maps imply surjectivity in $H^1$ since there is no $H^2$. If the dimension of $X$ is at least 2, then we continue the induction with $\ol{Z} = \phi, \ol{W} = 0$. So the conclusion is trivially true for $q = 0$. Then the conclusion is true for $q = 1$ by Rmk. \ref{surjective} (ii) and Lemma \ref{inductive lemma}. Then apply Lemma \ref{inductive lemma} repeatedly. 
\end{proof}

\section{Enlarged secant space}

In this section, our key conclusion (Propersition \ref{weights}) is the following. Recall that $L_d = dA$ for some very ample $A$.  Let $\mathrm{wts}(U)$ denote the torus weights in a toric representation $U$. In the setting of syzygies, we prove that there exist torus stable quotient spaces $W_{0,d}, W_{1,d}$:
$$H^0(X,L_d)  \xrightarrow\!\!\!\!\to W_{1,d} \xrightarrow\!\!\!\!\to W_{0,d},$$
with:
$$\dim (W_{0,d}) \in o(d^n), \ \mathrm{and} \ h^0(X,L_d) - W_{1,d} \in o(d^n),$$
such that for any $W_d$ with:
$$\dim W_d = p+q, \ \mathrm{and} \ W_{1,d} \xrightarrow\!\!\!\!\to W_d \xrightarrow\!\!\!\!\to W_{0,d},$$
we have:
$$\mathrm{weights}(\wedge^{p+q}W_d) + \mathrm{some \ fixed \ weight} \subset \mathrm{weights}(K_{p,q}(X;L_d)).$$
As explained in the proof outline in the Introduction, this will let us produce many different weights in $K_{p,q}(X;L_d)$. 

\begin{lem}\label{global generation}
Let $X$ be a scheme with $A$ a very ample divisor and $L_d = dA$. Let $Z$ be a subscheme and $\{E_i\}$ a collection of divisors such that $Z = \cap_i E_i$. Then there exists a subspace $\J_d \subset H^0(X,L_d)$ such that $\J_d$ generates $\I_{Z/X}(L_d)$ and $\dim J_d \in O(1)$. 
\end{lem}

\begin{proof}
If $\J^i_d \subset H^0(X,L_d)$ generates $\I_{E_i}(L_d)$, then $\sum_i \J^i_d$ generate $\I_{Z/X}(L_d)$, so we can assume $Z = E$ is a divisor. For some large $N$, $\I_E(L_N)$ is globally generated by sections $F_1,...F_{m_1}$. Assume $A$ is globally generated by sections $s_1,..,s_{m_2}$. Then for $d > N$, $\I_E(L_d) = \O_X(L_d-E)$ is globally generated by $\J_d = <F_i \cdotp s_j^{d-N}>$. It has finitely many vector space generators, so $\dim \J_d$ is finite. 
\end{proof}

\begin{rmk}
It is straightforward to see that the above lemma is still true if we start with torus equivariant objects and want torus equivariant $\J_d$'s.
\end{rmk}

\begin{prop}\label{big W}
There exist torus equivariant quotients $\W_{0,d}, \W_{1,d}$ of $H^0(X,L_d)$, such that 
\begin{equation}\label{variant of W}
\ol{W_{0,d}} = \ol{W_{1,d}} = \ol{W_d}, \quad Z = X \cap \P(\W_{0,d}) = X \cap \P(W_{1,d})
\end{equation}
and the dimensions satisfy
$$\dim W_{0,d} \in o(h^0(X,L_d))$$ 
and
$$\lim_{d \ra \infty} \frac{ \dim \W_{1,d}}{h^0(X,L_d)} = 1$$ Moreover, $\W_{0,d}$ and $\W_{1,d}$ carry weight $q$ syzygies for $\O_\pp$.
\end{prop}

\begin{proof}
Pick $\W_{0,d} = W_d$. This satisfies the conditions.  To construct $\W_{1,d}$, we will vary $W_d$ while keeping $Z$ and $\ol{W_{d}}$ the same, i.e. we look for large dimension quotients $W_{1,d}$ of $H^0(X,L_d)$ such that:
\begin{equation}\label{Z and W bar are the same 1}
Z = X \cap  \P(W_{1,d})
\end{equation}
and
\begin{equation}\label{Z and W bar are the same 2}
\ol{W_{1,d}} = \ol{W_{0,d}}.
\end{equation}
By the argument of Proposition \ref{carry syzygies}, if the $W_{1,d}$ satisfy the above conditions, then the $W_{1,d}$ also carry weight $q$ syzygies. We first constuct $W_{1,d}$ such that they satisfy the conditions in (\ref{Z and W bar are the same 1}) and (\ref{Z and W bar are the same 2}) and then do a dimension count. 

Let $W_{1,d}$ be the quotient of $V_d = H^0(X,L_d)$ by $ \mathcal{J}_{1,d}$, i.e. 
$$W_{1,d} = V_d / \mathcal{J}_{1,d}.$$
To satisfy (\ref{Z and W bar are the same 1}), $\J_{1,d}$ has to generate $\I_{Z/X}(L_d)$. By Lemma \ref{global generation}, we can pick torus equivariant $\J_{1,d}$ of bounded dimension. To satisfy (\ref{Z and W bar are the same 2}), for the consecutive quotient maps: 	
$$V_d \rightarrow \W_{1,d} \rightarrow W_{d},$$
we need:
\begin{equation}\label{J_d in lower}
V_d'+J_d = V_d'+ \J_{1,d}
\end{equation}
where 
$$V_d' = \ker(V_d \ra \ol{V_d}), \quad J_d = \ker(V_d \ra W_d).$$
We start with a toric basis of $V_d$, then $V'_d$, $J_d$ both have induced toric basis elements. Denote by $B(V'_d)$  the toric basis elements of $V'_d$, and   $B(J_d)$ those of $J_d$. Then we can choose $\J_{1,d}$ to have a basis containing $B(V'_d) - B(J_d)$, but contained in $B(V'_d) \cup B(J_d)$.  $\J_{1,d}$ will be toric equivariant and we have the dimension count:
\begin{equation*}
\begin{aligned}
\dim(V_d'+J_{d}) - \dim V_d' = \dim V_d - \dim \ol{W_d} - \dim V_d' \\
= \dim \ol{V_d} - \dim \ol{W_d} \\
\leq \dim \ol{V_d}.
\end{aligned}
\end{equation*}
Hence (\ref{J_d in lower}) requires the $\J_{1,d}$ to be appropriate subspaces of $V_d'+ \J_{d}$ with the following range of dimensions:
\begin{equation*}
\dim \ol{V_d} \ \leq \ \dim \J_{1,d}\  \leq \ \dim(V_d'+J_{d}).
\end{equation*}
Note that $\dim \ol{V_d} \in o(d^n)$. Therefore, to satisfy both  (\ref{Z and W bar are the same 1}) and (\ref{Z and W bar are the same 2}), we can choose $\J_{1,d}$ such that $\dim \J_{1,d} \in o(d^n) + O(1)$. Since $\W_{1,d} = V_d / \J_{1,d}$, $W_{1,d}$ will be torus equivariant and we have
$\dim \W_{1,d} = \dim V_d  - \dim \J_{1,d}$. Then $h^0(X,L_d) \in \Theta(d^n)$ imply that $$\lim_{d \ra \infty} \frac{ \dim \W_{1,d}}{h^0(X,L_d)} = 1.$$
\end{proof}
Finally, we conclude:
\begin{prop}\label{weights}
For any torus equivariant $\W_d$ that fits in the following diagram of consecutive equivariant quotient maps:
$$\W_{1,d} \ra \W_d \ra \W_{0,d}$$
we have:
\begin{equation}\label{contained weights}
\mathrm{wts}(\bigwedge^{\dim \W_{d}} \W_{d}) + \mathrm{wts}(H^q(X, I_{Z/X})) \subseteq \mathrm{wts}(K_{\dim \W_{d} - q,q}(\pp;L_d))
\end{equation}
\end{prop}
\begin{proof}
By the argument of Proposition \ref{carry syzygies} , $\W_d$ carries weight $q$ syzygies. Then the weight inclusions follow from Lemma \ref{weight inclusion}.
\end{proof}

\section{Asymptotics of normalized weights}

In this section, we prove the main result. As explained in the Introduction, one of the issues is to study torus weights of $\wedge^{p}W_d$ given the weights of $W_d$. Recall that we defined in the Introduction:
\[  \mathrm{wts}(K_{p,q}(X;L_d)) = \{ \mathrm{Torus \ weights \ of \ }K_{p,q}(X;L_d) \} \subseteq (p+q)d \cdotp \Delta,\] 
where $\Delta$ is the convex polytope associated with the very ample divisor $A$. 
We are interested in the normalized weights:
\[  \mathrm{wts^{nor}} (K_{p,q}(X;L_d)) = \frac{\mathrm{wts}(K_{p,q}(X;L_d)) }{(p+q)d} \subseteq \Delta\]
In this section, we work asymptotically and we are interested in asymptotic closures:
\[ \Delta(a, b) \  = \overline{\bigcup_{\substack{d  \gg  0 \\ a \cdot r_d  \leq   p  \leq   b \cdot  r_d}} \mathrm{wts^{nor}}(K_{p,q}(X;L_d))}\subseteq \ \Delta \]
Keeping notation from previous sections, the contributions from weights in $H^q(X,I_{Z/X})$ and $W_{0,d}$ will be asymptotically insiginificant to normalized weights, in other words:
\[  \Delta(a,b) \supseteq \overline{\bigcup_{\substack{d \  \gg \  0 \\ a \cdot r_d \ \leq \  p \ \leq \  b \cdot  r_d}} \mathrm{wts^{nor}}\left(\bigwedge^{\dim \W_{d}} \W_{d} + H^q(X, I_{Z/X})\right)} =  \overline{\bigcup_{\substack{d \  \gg \  0 \\ a \cdot r_d \ \leq \  p \ \leq \  b \cdot  r_d}} \mathrm{wts^{nor}}\left(\bigwedge^p W_{1,d}\right)}\]
Hence, to give a lower bound for $\Delta(a,b)$, we can simply work with the sequence $\{W_{1,d}\}$ in (\ref{contained weights}) and for simplicity, we abuse notation and write $W_d$ instead of $W_{1,d}$. 

After this simplication, we arrive at the following setting. 
Let $\Delta \subseteq \R^n$ be the convex polytope associated to a very ample divisor $A$ on $X$. For $d \in \mathbb{N}$, let 
$$\W_d \subset d\Delta \cap \Z^n$$ 
be a subset. 
Let us use $\bigwedge^{p_d} \W_d$ to denote the collection of points in $\Z^n$ expressible as nonrepetitive sums of $p_d$ points in $\W_d$. Assume that for $d \in \mathbb{N}$,
\begin{equation}\label{limit}
\mathrm{lim}_{d \rightarrow \infty} \frac{|\W_d|}{|d\Delta|} = 1
\end{equation}

Take any point $x$ inside the polytope $\Delta$. For simplicity, we call the finite union of cubes contained in $\Delta$ a shape.  Recall that:
$$\tau_x = \mathrm{sup} \ \{ \vol(\mathbb{S}_x)\  | \ \mathbb{S}_x \mathrm{\ is \  a\ shape \subset \Delta}, \mathrm{\ center \ of \ mass \ of} \ \mathbb{S}_x = x\}$$

The intuition for the two key lemmas (\ref{dense in a ball}, \ref {upper bound}) in this section are as follows. 
\begin{enumerate}
\item If there is a shape of volume $\eta_x$ centered at $x$, we can pick any number between 1 and $(\eta_x - \epsilon)d^n$ such that the the average of these weights lie arbitrarily close to $x$ asymptotically. 
\item Any sequence of subsets of lattice weights averaging to $x$ cannot exceed $\tau_x$ many elements. 
\end{enumerate}
\begin{lem}\label{dense in a ball}
With the above notations, for any sequence 
$$1 \leq p_d \leq (\eta_x - \epsilon)d^n$$
and any open set $U_x \subset \mathbb{R}^n$ containing $x$, there exists $d_0$ such that for all $d > d_0$, $U_x$ contains a point of 
$$\frac{\mathrm{wts}^{\mathrm{nor}}(\bigwedge^{p_d}\W_d)}{p_d \cdotp d}$$
\end{lem}

\begin{proof}[Sketch of Proof]
Denote by $S_x$ a finite union of cubes with volume $\eta_x$, center of mass $x$. If we take lattice points in multiples of $S_x$, take the average and normalize, then the average weight converges to $x$ by the definition of center of mass. If the $p_d$ are small and we take integral points in smaller balls, the same approximation works.
\end{proof}

\begin{thm}\label{dense}
Fix $1 \leq q \leq n$. Then
$$\bigcup_{\substack{d \ > \ 0 \\ 1 \leq \ p \leq \  r_d}} \mathrm{wts^{nor}}(K_{p,q}(X;L_d))$$ 
is dense in $\Delta$.
\end{thm}

\begin{proof}
By Proposition \ref{big W} and \ref{weights}, we know that $\mathrm{wts}^{\mathrm{nor}}(K_{p,q}(\P;L_d))$ contains weights corresponding to nonrepetitive sums of weights in $\W_{1,d}$. By Proposition \ref{big W} and Lemma \ref{dense in a ball}, they are dense in $\Delta$. 
\end{proof}

\begin{lem}\label{upper bound}
Take $y \in \Delta$. Assume
\begin{equation}\label{y as a limit}
y = \lim_{d \ra \infty} y_d
\end{equation}
where $y_d$ is an average of  $w_{d,j} \in (d \cdotp \Delta) \cap \Z^n$, i.e.:
$$ y_d = \frac{1}{i_d}\sum_{1 \leq j \leq i_d} \frac{w_{d,j}}{d}, \quad w_{d,j} \ \mathrm{distinct}, \quad \mathrm{and} \  i_d = |\{w_{d,j}\}| .$$
Then for any subsequence $\{i_{d_j}\}$:
$$\mathrm{limsup}_{j \ra \infty} \frac{i_{d_j}}{d^n} \leq \tau_y$$
\begin{proof} 
Suppose there is a subsequence $i_{d_j'}$ such that 
$$\mathrm{liminf}_{j \ra \infty} \frac{i_{d_j'}}{d^n} \geq \tau.$$
We claim that for any constant $\epsilon > 0$, there is a shape centered at $y$ with volume at least $\tau - \epsilon$. 
Assuming the claim, the lemma follows since if we had 
$$\mathrm{limsup}_{j \ra \infty} \frac{i_{d_j}}{d^n} = \tau > \tau_y,$$ then there is a subsequence $i_{d_j'}$ such that 
$$\mathrm{liminf}_{j \ra \infty} \frac{i_{d_j'}}{d^n} = \frac{\tau+\tau_y}{2} > \tau_y$$
By the claim, this is a shape with center of mass at $y$ and volume $\frac{\tau+\tau_y}{2}> \tau_y$. This is a contradiction to $\tau_y$ being the biggest volume supported at $y$.  

Let's turn to the claim. For convenience, we write $i_j$ for $i_{d_j}$. We construct a shape centered at $y$ with volume at least $\tau - \epsilon$ by taking small cubes centered at $w_{d_j,i}$, then make two adjustments: boundary adjustment and center adjustment. 

More specifically, for each $w_{d_j,i}$ with $1 \leq i \leq d_j$, take the $1 \times ... \times 1$ cubes of dimension $n$ centered at $w_{d_j,i}$. Denote each cube by $C_{d_j,i}$. Suppose the $2 \times ... \times 2$ cube centered at $w_{d_j,i}$ does not intersect the boundary of $d \Delta$ for $ 1 \leq i \leq i'_d$. The boundary of $d\Delta$ is bounded by $d (1-\epsilon) \Delta$ and $d (1+\epsilon) \Delta$ for large $d$ and any $\epsilon >0$. Hence, $i_d - i'_d \in o(d^{n})$.

Take the union:
$$\Sigma'_d= \bigcup_{i = 1}^{i'_d} C_{d_j,i}$$
and denote by $y'$ the center of mass for $\Sigma'_d$. $\Delta$ is bounded, so by the previous paragraph and assumption:
$$|y' - y| \leq |y' - y_d| + |y_d - y| \in o(d^n) + O(1).$$

Move $\Sigma'_d$ by: 
$$\frac{y - y'}{i'_d}$$
Then it is straightforward to see that for large $d$, the shape above is contained in $\Delta$ and has volume arbitrarily close to $\tau$. 
\end{proof}
\end{lem}

Recall that we define:
\[ \Delta (a, b) \  = \overline{\bigcup_{\substack{d  \gg   0 \\ a \cdot r_d  \leq p  \leq  b \cdot  r_d}} \mathrm{wts^{nor}}(K_{p,q}(\PP;L_d))}\subseteq \ \Delta \]

\begin{thm}\label{delta a b} One has:
$$\Delta(a,b) = \left \{x \in \Delta \  \bigg| \  \frac{\tau_x}{\vol(\Delta)} \geq a \right \} =: \Delta(a)$$ 
\end{thm}

\begin{proof}
The closure 
$\Delta(a,b) \supseteq \left \{x \in \Delta \  \bigg| \  \frac{\tau_x}{\vol(\Delta)} \geq a \right \}$ follows from a similar argument as in Lemma \ref{dense} replacing a sphere with shapes supported at $x$ with volumes arbitrarily close to $\tau_x$. For sequences asymptotically small, we can pick points in a sphere around $x$. For sequences asympotically close to $\tau_x$, we can pick points from finite cube unions (possibly further scaled down) approximating $\tau_x$.

As is well known (first in \cite{G} (0.1)-(0.4)), $K_{p,q}(\pp;L_d)$ can be computed as cohomology in the middle of the following short complex:
\begin{align*}
\wedge^{p+1}H^0(\pp,L_d) \otimes H^0(\pp, (q-1)L_d) \ra \wedge^{p}H^0(\pp,L_d) \otimes H^0(\pp, q L_d) \\
\ra \wedge^{p-1}H^0(\pp,L_d) \otimes H^0(\pp, (q+1)L_d)
\end{align*}
Hence all weights of $K_{p,q}(\pp,L_d)$ correspond to nonrepetitive sums of points in $(p+q)\Delta$. The length of the above Koszul complex is $h^0(\pp,L_d) = \vol(\Delta) \cdotp d^n + O(d^{n-1})$. $q$ is bounded, so:
$$\overline{\bigcup_{\substack{d \  \gg \  0 \\ a \cdot r_d \ \leq \  p \ \leq \  b \cdot  r_d}} \mathrm{wts^{nor}}(\wedge^{p+1}H^0(\pp,L_d) \otimes H^0(\pp, (q-1)L_d))}$$
$$= \overline{\bigcup_{\substack{d \  \gg \  0 \\ a \cdot r_d \ \leq \  p \ \leq \  b \cdot  r_d}} \mathrm{wts^{nor}}(\wedge^{p+1}H^0(\pp,L_d))}.$$
and similarly for the other terms in the Koszul complex. 
By Lemma \ref{upper bound}, if a normalized weight is in 
$$\overline{\mathrm{wts}^{\mathrm{nor}}(\wedge^{p_d}H^0(\pp,L_d))},$$
then it satisfies 
$$\tau_x \geq \mathrm{limsup}_{d \ra \infty} \frac{p_d}{d^n}$$
We are computing for the range $[a \vol(\Delta)d^n, b \vol(\Delta)d^n]$, hence 
$$\mathrm{limsup}_{d \ra \infty} \frac{p_d}{d^n} \geq b \vol(\Delta).$$
so $$x \geq a \vol(\Delta).$$
Then by the definition of $K_{p,q}(\pp,L_d)$ above: 
\begin{align*}
\ol{\mathrm{wts}^\mathrm{nor} (K_{p,q}(\pp,L_d))} 
 & \subseteq \ol{\mathrm{wts}^\mathrm{nor}(\wedge^pH^0(\pp,L_d) \otimes H^0(\pp,qL_d))} \subseteq \left \{x \in \Delta \  \bigg| \  \frac{\tau_x}{\vol(\Delta)} \geq a \right \}
\end{align*}
which gives us the other inclusion.
\end{proof}

\section{Boundary of $\Delta(a)$}

In this section, we describe the boundary of $\Delta(a)$.

Fix a convex body $\Delta \subset \R^n$ with a nice (eg. piecewise polynomial) boundary, and a constant $a$ in $[0,1]$. Let $v$ be a unit vector in $\R^n$. Then $H(v,c) := \{ x \in \R^n | v.x \leq c \}$ forms a family of parallel half spaces. There is a unique constant $c$ so that:
$$\vol(\Delta \cap H(v,c)) = a \vol(\Delta).$$
Call this constant $c_v$. Let $x_v$ be the center of gravity of $\Delta \cap H(v,c_v)$. 

\begin{prop}\label{boundary}
The points $x_v$, as v ranges over all unit vectors, form the boundary of $\Delta(a)$. 
\end{prop}

\begin{rmk}
We give the intuition but skip the proof of the proposition. Think of $\Delta$ as a container holding water of volume $a$. When one leans it to the direction with $v$ pointing downward, water flows to lower its center of mass (potential energy). In this position, the center of mass, call it $x_v$, is the extreme of $\Delta$ in direction $v$, hence a boundary point of $\Delta(a)$. The water level in this setting corresponds to the boundary of $H(v,c_v)$ and we can actually see that the tangent of the boundary is the hyperplane perpendicular to $v$ through $x_v$. 
\end{rmk}

\begin{exap}
In the example given in the introduction, when $\Delta$ is the unit square and $a = \frac{1}{10}$. The interested reader can work out that when the water surface divides the square into a triangle and a pentagon, the center of mass of the water is parametrized by 
$$(1- \frac{1}{15k}, \frac{1}{3}k) \ \mathrm{for} \ \frac{1}{5} \leq k \leq 1$$
and its symmetric images, these account for the 4 hyperbola segments. 
When the water surface divides the unit square into two trapezoids, the center of mass lies on:
$$(\frac{2}{3}  - \frac{5}{3} k, \frac{29}{30} - \frac{1}{6} k)\ \mathrm{for} \ \frac{1}{10} \leq k \leq \frac{1}{5}$$ 
or its symmetric images. These account for the 8 line segments. 
\end{exap}

\end{document}